\documentclass{amsart}
\usepackage[utf8]{inputenc}
\usepackage{amsfonts}
\usepackage{amsmath}
\numberwithin{equation}{section}
\usepackage{amssymb}
\usepackage{float}
\usepackage{tikz-cd}
\usepackage{amsthm}
\usepackage{enumerate}
\usepackage{tikz}
\usepackage{graphicx}
\usepackage{hyperref}
\usepackage[capitalise]{cleveref}
\usepackage{slashed}  % For letters with a slash
\usepackage{upgreek}
\usepackage[top=3cm,bottom=3cm,left=2.5cm,right=2.5cm,headsep=10pt,letterpaper]{geometry} % Page margins
\usepackage{fancyhdr}

\usepackage[
    backend=biber,
    style=alphabetic,
    sorting=nyt
]{biblatex}
\numberwithin{equation}{section}
\theoremstyle{plain}
\newtheorem*{claim}{Claim}

\newtheorem{theorem}{Theorem}[section]
\newtheorem{proposition}[theorem]{Proposition}
\newtheorem{lemma}[theorem]{Lemma}
\newtheorem{corollary}[theorem]{Corollary}

\theoremstyle{definition}

\theoremstyle{remark}
\newtheorem{remark}[theorem]{Remark}

\newcommand{\Sph}{\mathbb{S}^{2}}

\newcommand{\diam}{\operatorname{diam}}

\newcommand{\RR}{\mathbb{R}}

\newcommand{\MM}{\mathbb{M}}
\DeclareMathOperator{\Ric}{Ric}
\DeclareMathOperator{\sn}{sn}
\DeclareMathOperator{\cs}{cs}
\DeclareMathOperator{\ct}{ct}
\DeclareMathOperator{\csch}{csch}

\title[A sharp lower bound for the First Nonzero Neumann Eigenvalue]{A Sharp Diameter-Dependent Lower Bound for the First Nonzero Neumann Eigenvalue of Geodesic Triangles in Space Forms}

\author{Shoo Seto}
\address{Department of Mathematics, California State University, Fullerton, CA 92834}
\email{shoseto@fullerton.edu}

\author{Guofang Wei}
\address{Department of Mathematics, University of California, Santa Barbara, CA 93106}
\email{wei@math.ucsb.edu}
\thanks{G. Wei is partially supported by NSF DMS 2403557.}

\author{Yusen Xia}
\address{Department of Mathematics, University of California, Santa Barbara, CA 93106}
\email{yusen@ucsb.edu}
\thanks{Y. Xia is partially supported by NSF DMS 2403557.}

\begin{document}

\begin{abstract}
We prove a sharp lower bound for the first nonzero Neumann eigenvalue of geodesic triangles of given diameter in two-dimensional space forms. The bound is given by the first positive radial Neumann eigenvalue of an one dimensional model; it is approached by degenerating isosceles triangles. When \(K>0\) and \(D=\pi/(2\sqrt K)\), equality is attained precisely by birectangular triangles. We also prove a hot-spots theorem for non-acute spherical triangles of diameter at most \(\pi/2\), and establish antisymmetry and eigenvalue monotonicity for isosceles spherical triangles of diameter \(\pi/2\).
\end{abstract}
\maketitle

\tableofcontents

\section{Introduction}
Given a Riemannian manifold $(M,g)$ and a connected bounded domain $\Omega \subset M$ with piecewise smooth boundary we consider the Neumann eigenvalue problem of the Laplacian
\begin{equation*}
    \begin{cases}
        \Delta u + \mu u=0 &\text{ in }\Omega\\
        \nabla_\nu u = 0 &\text{ on }\partial \Omega,
    \end{cases}
\end{equation*}
where $\nabla_\nu$ denotes the (outward) normal derivative defined at $C^1$ points of $\partial\Omega$.  For bounded domains the spectrum is discrete and its eigenvalues can be ordered as 
\begin{align*}
    0=\mu_0 <\mu_1 \leq \mu_2\leq \mu_3\leq \cdots \to \infty.
\end{align*}
The eigenvalues of the Laplacian provide a fundamental link between analysis and geometry. They describe the characteristic modes of diffusion and vibration, while the low eigenvalues control quantitative properties such as decay to equilibrium and Poincaré inequalities. At the same time, these eigenvalues depend sensitively on the geometry of both the ambient space and the domain, including its dimension, scale, shape, curvature, and boundary geometry. A central problem in spectral geometry is therefore to determine which geometric constraints yield sharp bounds for the spectrum and to identify the domains for which equality occurs.

The first positive Neumann eigenvalue of a bounded domain $\Omega \subset M$ is given variationally as
\begin{equation*}
    \mu_1^N(\Omega):=\inf_{\substack{u\in H^1(\Omega)\backslash\{0\}\\\int_\Omega udV=0}} \frac{\int_{\Omega}|\nabla u|^2dV}{\int_\Omega u^2dV}.
\end{equation*}
Note that, unlike in the Dirichlet case ($u=0$ on $\partial \Omega$), the relation between the space of admissible functions for subsets $\Omega'\subset\Omega \subset M$ is not immediately clear; thus the Neumann eigenvalues do not (necessarily) have the domain monotonicity property, not even for triangles; see \cite{LaugesenSiudeja2010,FreitasKennedy2025}.

For the Neumann problem, since nonzero constants are eigenfunctions with eigenvalue zero,
\begin{align*}
    \mu_1^N-\mu_0^N=\mu_1^N
\end{align*} is called the fundamental gap.  Outside of highly symmetric domains, the gap cannot be computed explicitly.

Obtaining sharp lower bounds to the gap has been an important program within geometric analysis.  For bounded convex Euclidean domains $\Omega$ with diameter $D$, the Payne-Weinberger inequality \cite{PayneWeinberger1960} gives a lower bound 
\begin{equation}\mu_1^N(\Omega)\geq \frac{\pi^2}{D^2}.
\end{equation}
The constant is approached by domains degenerating to a line segment. In general for bounded convex domains of manifolds with $\Ric \geq (n-1)K$, there is lower bound given by one dimensional comparison model \cref{eq:convex-model-comparison}, 
\begin{equation*}
    \mu_1^N\geq \bar{\mu}_K(D) > \frac{\pi^2}{D^2}+\frac{n-1}{2}K,
\end{equation*}
see \cite{BenAndrews2013}, also \cite{Kroger1992SpectralGap, BakryQian2000}.

Restricting the class of domains to triangles gives a strictly stronger constant.  For every nondegenerate Euclidean triangle $T$ of diameter $D$, the result of Laugesen and Siudeja \cite{LaugesenSiudeja2010} gives the estimate
\begin{align}
    \mu_1^N(T) > \frac{j_{1,1}^2}{D^2}  \label{mu-lower bound for R2}
\end{align}
where $j_{1,1} \approx 3.83$ is the first positive root of the Bessel function $J_1$. Equality is not attained by any nondegenerate triangle and is approached by a sequence of acute isosceles triangles degenerating to a line segment.   Laugesen-Siudeja also show an isosceles symmetry transition.  For an isosceles Euclidean triangle, the fundamental Neumann eigenfunction is symmetric across the axis when the aperture is less than $\pi/3$, and antisymmetric when the aperture is greater than $\pi/3$.  At the equilateral triangle, the first positive eigenvalue has multiplicity two.

In this paper, we treat the cases $\mathbb M^2_K$, the complete simply connected space form with constant curvature $K$, and $T \subset \MM^2_K$ a geodesic triangle, and give an optimal lower bound for the Neumann fundamental gap:

\begin{theorem}
\label{The Main Theorem}
Let $\MM_K^2$ be the complete simply connected two-dimensional
Riemannian manifold of constant sectional curvature $K\in\mathbb R$,
and let $T\subset \mathbb M_K^2$ be a nondegenerate geodesic triangle with $
D:=\operatorname{diam}(T).$
Assume
\[
D \in (0, \infty) \ 
\text{if }K\leq0, \ \ \ \ \ \ \ 
D \in (0, \frac{\pi}{2\sqrt K}] \ \text{if }K>0.
\]
Let $\nu_K(D)$ denote the first positive eigenvalue of the radial
Neumann problem
\[
-\bigl(\operatorname{sn}_K(r)J'(r)\bigr)'
=
\nu_K(D)\operatorname{sn}_K(r)J(r),
\qquad
0<r<D,
\]
with
\[
J'(0)=J'(D)=0,
\]
where $\operatorname{sn}_K(r)$ is defined in \eqref{eqn:sn_k}.

Then
\begin{equation}
\mu_1(T)\geq\nu_K(D).  \label{eq: mu lower bound}
\end{equation}

If either $K\leq0$, or $
K>0 \ \text{and} \ 
D<\frac{\pi}{2\sqrt K},$
then the inequality \eqref{eq: mu lower bound} is strict.

Moreover, the bound is sharp:
\[
\inf_{\operatorname{diam}(T)=D}\mu_1(T)
=
\nu_K(D).
\]
\begin{enumerate}
    \item If $K \leq 0$, or if $K>0$ and $D<\pi /(2 \sqrt{K})$, the infimum is approached by isosceles triangles whose two equal sides have length $D$ and whose apex angle tends to zero.
    \item If $K>0$ and $D=\frac{\pi}{2\sqrt K},$
then
\[
\nu_K(D)=6K,
\]
and equality $\mu_1(T)=6K$ holds if and only if $T$ is birectangular, namely, some vertex is
joined to the other two vertices by sides of length
$\pi/(2\sqrt K)$.
\end{enumerate}

\end{theorem}

This theorem extends the sharp result of Laugesen--Siudeja from Euclidean triangles to spherical and hyperbolic triangles and places all three geometries under a unified comparison principle. By setting $K=0$ we immediately recover \eqref{mu-lower bound for R2}, see \cref{rem:Euclidean-specialization}. Our proof is completely different. We  do not rely on deformations in $\RR^2$ or symmetry of the eigenfunctions of isosceles triangle, but use direct comparison and topology of nodal set.

We also obtain an explicit lower bound for $\nu_K(D)$ in \cref{eq: nu explicite lower bound} and \cref{nu lower bound K<0}. 

As $\nu_K(D)$ is monotone decreasing in $D$ (see \cref{lem:nu-monotone}), or using \cref{eq: nu explicite lower bound}, we immediately have the following in the case $K=1$. 
\begin{corollary}
Let $T\subset\mathbb S^2$ be a nondegenerate geodesic triangle with $D:=\operatorname{diam}(T) \le \frac{\pi}{2}.$
Then
\[
\mu_1(T) \ge 6,
\]
and the equality holds iff $T$ is birectangular. 

\end{corollary}

\begin{remark}
 The above estimate has applications to the boundary regularity of the Neumann harmonic function in a  conical region in $\mathbb R^3$,  which leads to the estimate of the singular boundary set of a free boundary area minimizing hypersurface in domains with dihedral angles everywhere not larger than $\pi/2$ \cite{LiEdelen2022}. 
\end{remark}

A second theme of the paper concerns the geometry of the first nontrivial Neumann eigenfunctions.  Rauch's hot-spots conjecture asserts that the extrema of such an eigenfunction occur on the boundary of the domain. Although the conjecture is false for general planar domains \cite{BW99}, it has been established under a variety of geometric assumptions; see, for example, \cite{BB99,AB04}. For Euclidean triangles, the conjecture was established by Judge and Mondal \cite{JudgeMondal2020, JudgeMondal2022Erratum}, and more recently Chen, Gui, and Yao \cite{ChenGuiYao2026} further obtained a precise description of the critical points and monotonicity of first nontrivial Neumann eigenfunctions on Euclidean triangles. 

In curved space, Hatcher \cite{Hatcher2025} developed an approach to hot spots on domains of constant curvature and proved that a first positive Neumann eigenfunction on a non-acute hyperbolic triangle has no non-vertex critical points and is strictly monotone along an appropriate Killing field. In \cite{Hatcher2025} it was observed that the argument would also extend to positive curvature provided one could establish the spherical analog of a certain comparison between the first positive Neumann eigenvalue and a mixed Dirichlet-Neumann eigenvalue.  We provide this bridge in Proposition \ref{Spherical Hatcher lemma 5.2}, resolving the precise positive-curvature obstruction identified in his Theorem 1.5.  With some modifications of the argument given in \cite{Hatcher2025}, we obtain Hot Spot Theorem for non-acute spherical triangles:

\begin{theorem}\label{sphericalHatcherTheorem}
Let $T \subset \Sph$ be a non-acute spherical geodesic triangle with $\diam (T)\leq\pi / 2$, not belonging to the exceptional birectangular family. Then:
\begin{itemize}
\item[1.] $\mu_1^N(T)$ is simple;
\item[2.] the first nonconstant Neumann eigenfunction $u$ has no non-vertex critical points and the global extremals are located at the vertices of the longest side;
\item[3.] if $e$ is the longest side, there exists a spherical Killing field $X_e$, tangent to $e$, such that, up to a sign change of $u$,
$$
X_e u>0 \quad \text { in } T .
$$
\end{itemize}
\end{theorem}

Using the Hot Spot Theorem, we are able to show the geometry of the first nontrivial Neumann eigenfunction. In particular, we extend the antisymmetry result in \cite{LaugesenSiudeja2010} and obtain a monotonicity result on $\mu_1$ for isosceles triangles of diameter $\pi/2:$

\begin{theorem}\label{Antisymmetry of isosceles triangles Theorem}
    Let $T\subset \Sph$ be a non-birectangular isosceles triangle with diameter $D =\pi/2.$ Then its first nonconstant Neumann eigenfunction $u$ is antisymmetric about its median. Moreover, as the apex vertex moves away from the midpoint of the base along the median, its first nonzero Neumann eigenvalue decreases.
\end{theorem}

\subsection*{Organization of the paper}
In \S \ref{sec: prelim}, we collect some basic facts used for proving \cref{The Main Theorem}, including elementary geometric properties of geodesic triangles in two-dimensional space forms and the local and global structure of nodal sets of Neumann eigenfunctions. In \S \ref{sec:space-form-diameter-bound}, we prove the sharp diameter-dependent lower bound for geodesic triangles in two-dimensional space forms. The proof combines a radial Neumann comparison function with the nodal topology of a first Neumann eigenfunction. We also establish sharpness by a degeneration to thin isosceles triangles. We then give an estimate on the lower bound in terms of curvature and diameter. In \S \ref{section:Spherical Hatcher Theorem} we prove the spherical mixed-eigenvalue comparison and deduce the non-acute spherical hot-spots theorem. Finally, \S \ref{Section: lower bound isoceles} treats the isosceles case, where we determine the symmetry of the first Neumann eigenfunction and study the variation and monotonicity of the corresponding eigenvalue. 

\subsection*{Acknowledgments} The second author thanks Prof. Mark Ashbaugh for sending the work of  Harvey Walden in the 1970s on the numerical algorithem computing the first Dirichlet eigenvalue of domains on  the sphere $\mathbb S^2$  in May 2026. Although we do not use these papers, it leads the authors to come back to this project which we explored a few years ago but stopped. The second author also thanks Chao Li for sharing that Neumann eigenvalues estimates of triangles in sphere is related to the  regularity of free boundary minimal surfaces. 

\subsection*{AI disclosure statement}The authors used OpenAI's ChatGPT as an assistive tool in the mathematical development and preparation of this manuscript. In particular, ChatGPT helped to formulate the comparison model and the proof strategy for the main eigenvalue estimate. It was also used to explore auxiliary arguments, pre-check calculations, and improve the exposition. All AI-assisted suggestions were independently checked, developed, and verified by the authors, who take full responsibility for the content of the paper.

\section{ Notations and Basic Facts}\label{sec: prelim}
Let $ \MM_K^2$ denote the complete simply connected two-dimensional
space form of constant sectional curvature $K\in\mathbb R$.
Introduce the standard functions
\begin{equation}  \label{eqn:sn_k}
\sn_K(r)
=
\begin{cases}
\dfrac{1}{\sqrt K}\sin(\sqrt K\,r), & K>0,\\[2mm]
r, & K=0,\\[2mm]
\dfrac{1}{\sqrt{-K}}\sinh(\sqrt{-K}\,r), & K<0,
\end{cases}
\end{equation}
and
\[
\cs_K(r):=\sn_K'(r)
=
\begin{cases}
\cos(\sqrt K\,r), & K>0,\\
1, & K=0,\\
\cosh(\sqrt{-K}\,r), & K<0.
\end{cases}
\]
Thus
\[
\sn_K''+K\sn_K=0,
\qquad
\cs_K^2+K\sn_K^2=1.
\]

In geodesic polar coordinates centered at a point of $\MM_K^2$,
\[
g=dr^2+\sn_K^2(r)\,d\theta^2,
\]
and
\[
\Delta
=
\partial_{rr}
+
\frac{\cs_K(r)}{\sn_K(r)}\partial_r
+
\frac{1}{\sn_K^2(r)}\partial_{\theta\theta}.
\]

For $K>0$ we restrict throughout to the hemisphere
\[
0<D\leq\frac{\pi}{2\sqrt K},
\]
while for $K\leq0$ we allow arbitrary $D>0$.

\subsection{The geometry of geodesic triangles}

We next record the geometric ingredients that will be used later in the proof.

\begin{lemma}
\label{lem:space-form-opposite-side-distance}
Let
\[
T=\triangle ABC\subset \MM_K^2
\]
be a nondegenerate geodesic triangle with $
\diam(T)=D. $
Assume either
\[
K\leq0, \qquad \mbox{or} \qquad 
K>0, \ 
D\leq\frac{\pi}{2\sqrt K}.
\]
Then for every $X\in BC^\circ,$
one has
\[
d(A,X)<D,
\]
except possibly when
\[
K>0,
\qquad
D=\frac{\pi}{2\sqrt K},
\qquad
AB=AC=D.
\]
In this exceptional case
\[
d(A,X)=D
\qquad
\text{for every }X\in BC.
\]
\end{lemma}

\begin{proof}
We treat the three curvature signs separately.

Suppose first that $K=0$.  Since the Euclidean norm is strictly
convex along a segment not containing $A$,
\[
|A-X|
<
\max\{|A-B|,|A-C|\}
\leq D,
\qquad
X\in BC^\circ.
\]

Now suppose $K\neq0$.  Write
\[
b=d(B,C),
\qquad
s=d(B,X),
\qquad
0<s<b.
\]
The standard interpolation formula along the geodesic segment
$BC$ gives
\begin{equation}
\label{eq:space-form-geodesic-interpolation}
\cs_K(d(A,X))
=
\frac{\sn_K(b-s)}{\sn_K(b)}\cs_K(AB)
+
\frac{\sn_K(s)}{\sn_K(b)}\cs_K(AC).
\end{equation}

Suppose $K>0$.  Since
\[
AB,AC\leq D\leq\frac{\pi}{2\sqrt K},
\]
the function $\cs_K$ is nonnegative and decreasing, we have $
\cs_K(AB),\ \cs_K(AC)\geq\cs_K(D).$
Moreover,
\[
\sn_K(b-s)+\sn_K(s)>\sn_K(b)
\]
for $0<s<b<\pi/\sqrt K$.

If $
D<\frac{\pi}{2\sqrt K},$
then $\cs_K(D)>0$, and hence \cref{eq:space-form-geodesic-interpolation} gives
\[
\cs_K(d(A,X))
>
\cs_K(D).
\]
Since $\cs_K$ is strictly decreasing on
$[0,\pi/\sqrt K]$, we conclude that
\[
d(A,X)<D.
\]

If $
D=\frac{\pi}{2\sqrt K}, $
then $
\cs_K(D)=0.$
The right-hand side of
\cref{eq:space-form-geodesic-interpolation} is strictly positive
unless
\[
\cs_K(AB)=\cs_K(AC)=0,
\]
that is,
\[
AB=AC=D.
\]
In the exceptional case the right-hand side vanishes identically,
so $
d(A,X)=D
$
for every $X\in BC$.

Finally, suppose $K<0$.  Then $\cs_K$ is strictly increasing and
\[
\cs_K(AB),\cs_K(AC)\leq\cs_K(D).
\]
Furthermore,
\[
\sn_K(b-s)+\sn_K(s)<\sn_K(b),
\qquad
0<s<b.
\]

Thus
\[
\cs_K(d(A,X))
<
\cs_K(D),
\]
and strict monotonicity of $\cs_K$ gives
\[
d(A,X)<D.
\]
\end{proof}

Note that the above lemma is not true anymore when $K>0, \ D > \tfrac{\pi}{2\sqrt{K}}$. 

The following is another simple fact about the distance function from a vertex of a triangle. 
\begin{lemma}
\label{lem:space-form-outward-radial}
Let $A$ be a vertex of a nondegenerate geodesic triangle
$T=\triangle ABC$, and let
\[
r(X)=d(A,X).
\]
Then on the open opposite side $BC^\circ$,
\[
\partial_\nu r>0,
\]
where $\nu$ denotes the outward unit conormal of $T$.
\end{lemma}

\begin{proof}
Fix $X\in BC^\circ$.  The minimizing geodesic from $X$ to $A$
enters the interior of $T$, since $A$ lies strictly on the inward
side of the geodesic containing $BC$.  Its initial velocity at $X$
therefore has a strictly positive inward conormal component.

On the other hand, $-\nabla r$ is precisely the unit tangent at $X$
pointing along the minimizing geodesic from $X$ toward $A$.
Therefore $-\nabla r$ has positive inward conormal component, or
equivalently,
\[
\langle\nabla r,\nu\rangle>0.
\]
Thus $
\partial_\nu r>0. $
\end{proof}

\subsection{Topology of the nodal set}

We shall use the following elementary nodal-topology fact.

\begin{lemma}
\label{lem:space-form-three-side-sign}
Let $u$ be a nontrivial Neumann eigenfunction on a geodesic triangle
$T\subset \MM_K^2$.  If $u$ changes sign on the relative interior of
each of the three sides, then $u$ has at least three nodal domains.
\end{lemma}

\begin{proof}
The nodal set is a finite embedded graph in the closed topological
disk $\overline T$ \cite[Theorem~2.5]{Cheng1976}.  At an interior zero of $u$, at least four
nodal arcs meet, while at a boundary zero at least one nodal arc
meets the boundary.

Let $V_{\mathrm{int}}$ and $V_\partial$ denote the interior and
boundary vertices of the nodal graph, respectively. Denote $N(u)$ the number of nodal domains of $u$. If a component of the nodal set is a simple closed curve containing no vertex, choose an arbitrary point on it and regard that point as a degree-two vertex. Euler's formula then gives,
\[
N(u)
=
1+c_0
+
\frac12
\sum_{v\in V_{\mathrm{int}}}(d(v)-2)
+
\frac12
\sum_{v\in V_\partial}d(v),
\]
where $c_0\geq0$ is the number of components of the nodal graph that
do not meet the boundary.

Since $u$ changes sign on each open side, the nodal set meets each
of the three sides.  Hence
\[
\#V_\partial\geq3
\]
and therefore
\[
\sum_{v\in V_\partial}d(v)\geq3.
\]
All other terms are nonnegative, so
\[
N(u)\geq1+\frac32.
\]
Since $N(u)$ is an integer,
\[
N(u)\geq3.
\]
\end{proof}

It is a standard fact that the first nonconstant Neumann eigenfunction changes sign exactly once.
\begin{lemma}
\label{lem:space-form-first-Neumann-two-domains}
Every eigenfunction corresponding to the first nonzero Neumann
eigenvalue of a geodesic triangle has exactly two nodal domains.
\end{lemma}

\section{A sharp eigenvalue lower bound for triangles in space forms}
\label{sec:space-form-diameter-bound}

\subsection{The radial Neumann model}

For an admissible $D$, let $\nu_K(D)$ denote the first positive
radial Neumann eigenvalue of the geodesic ball of radius $D$ in
$\MM_K^2$.  Equivalently, $\nu_K(D)$ is the first positive eigenvalue
of
\begin{equation}
\label{eq:space-form-radial-Neumann}
-(\sn_K(r)J'(r))'
=
\nu_K(D)\sn_K(r)J(r),
\qquad
0<r<D,
\end{equation}
with
\[
J'(0)=J'(D)=0.
\]
Equivalently,
\begin{equation}
\label{eq:space-form-radial-Rayleigh}
\nu_K(D)
=
\min_{\substack{
f\in H^1((0,D),\sn_K(r)\,dr)\\
\int_0^D f(r)\sn_K(r)\,dr=0}}
\frac{
\displaystyle\int_0^D |f'(r)|^2\sn_K(r)\,dr
}{
\displaystyle\int_0^D f(r)^2\sn_K(r)\,dr
}.
\end{equation}

For $\lambda>0$, let $J_\lambda$ be the solution of
\begin{equation}
\label{eq:space-form-Jlambda}
J_\lambda''
+
\frac{\cs_K}{\sn_K}J_\lambda'
+
\lambda J_\lambda
=
0,
\qquad
J_\lambda(0)=1,
\qquad
J_\lambda'(0)=0.
\end{equation}

\begin{lemma}[Radial monotonicity]
\label{lem:space-form-radial-monotonicity}
Let $D$ be admissible as above.

If $\lambda=\nu_K(D)$, then
\[
J_\lambda'(r)<0,
\qquad
0<r<D,
\]
and
\[
J_\lambda'(D)=0.
\]
If
\[
0<\lambda<\nu_K(D),
\]
then
\[
J_\lambda'(r)<0,
\qquad
0<r\leq D.
\]
\end{lemma}

\begin{proof}
Let
\[
\nu=\nu_K(D),
\qquad
J=J_\nu.
\]
By Sturm oscillation, $J$ has exactly one zero
\[
r_0\in(0,D).
\]
Choose its sign so that
\[
J>0\quad\text{on }(0,r_0),
\qquad
J<0\quad\text{on }(r_0,D).
\]
Since
\[
(\sn_K J')'
=
-\nu\sn_K J,
\]
and $J'(0)=0$, we obtain
\[
J'(r)<0
\qquad
0<r\leq r_0.
\]
On $(r_0,D)$, the quantity $\sn_K(r)J'(r)$ is strictly increasing.
Since $J'(D)=0$, it follows that
\[
J'(r)<0
\qquad
r_0\leq r<D.
\]

Now assume $0<\lambda<\nu$ and set
\[
Y_\lambda=-J_\lambda',
\qquad
Y_\nu=-J_\nu'.
\]
Differentiating \cref{eq:space-form-Jlambda} and using
\[
\left(\frac{\cs_K}{\sn_K}\right)'
=
-\frac{1}{\sn_K^2},
\]
we obtain
\begin{equation}
\label{eq:space-form-Y}
(\sn_KY_\lambda')'
+
\left(
\lambda\sn_K-\frac{1}{\sn_K}
\right)Y_\lambda
=
0.
\end{equation}
The same equation holds for $Y_\nu$ with $\lambda$ replaced by
$\nu$.

Near $r=0$,
\[
Y_\lambda(r)
=
\frac{\lambda}{2}r+O(r^3),
\qquad
Y_\nu(r)
=
\frac{\nu}{2}r+O(r^3).
\]
Define
\[
W(r)
=
\sn_K(r)
\left(
Y_\lambda'Y_\nu-Y_\nu'Y_\lambda
\right).
\]
Then
\begin{equation}
\label{eq:space-form-Wronskian}
W'(r)
=
(\nu-\lambda)\sn_K(r)Y_\lambda(r)Y_\nu(r).
\end{equation}

Suppose that $Y_\lambda$ has a first zero
$r_1\in(0,D]$.  Then
\[
Y_\lambda>0
\qquad
\text{on }(0,r_1),
\]
while
\[
Y_\nu>0
\qquad
\text{on }(0,D).
\]
Hence
\[
W(r_1)>0.
\]
If $r_1<D$, then
\[
W(r_1)
=
\sn_K(r_1)Y_\lambda'(r_1)Y_\nu(r_1)
\leq0,
\]
a contradiction.  If $r_1=D$, then
\[
Y_\lambda(D)=Y_\nu(D)=0,
\]
so $W(D)=0$, again a contradiction.  Thus
\[
Y_\lambda>0
\qquad
0<r\leq D,
\]
which proves the result.
\end{proof}

\subsection{Proof of \cref{The Main Theorem}}

\subsubsection{Proof of the lower bound and rigidity}
Let
\[
T=\triangle ABC,
\qquad
\mu=\mu_1(T),
\]
and let $u$ be a corresponding first nonconstant Neumann
eigenfunction.

We first consider either $K\leq0$, or $
K>0,
\ 
D<\frac{\pi}{2\sqrt K}.$

Suppose, toward a contradiction, that
\[
\mu\leq\nu_K(D).
\]

Fix the vertex $A$ and define
\[
V_A(X)=J_\mu(d(A,X)).
\]
Since $J_\mu'(0)=0$, the radial function $V_A$ is smooth at $A$.
By \cref{eq:space-form-Jlambda},
\[
-\Delta V_A=\mu V_A.
\]
Since $V_A$ is radial about $A$,
\[
\partial_\nu V_A=0
\]
on $AB\cup AC$.

For $X\in BC^\circ$,
\cref{lem:space-form-opposite-side-distance} gives
\[
d(A,X)<D.
\]
By \cref{lem:space-form-radial-monotonicity},
\[
J_\mu'(d(A,X))<0,
\]
and by \cref{lem:space-form-outward-radial},
\[
\partial_\nu d(A,X)>0.
\]
Therefore
\begin{equation}
\label{eq:space-form-negative-flux}
\partial_\nu V_A
=
J_\mu'(d(A,X))
\partial_\nu d(A,X)
<0
\qquad
\text{on }BC^\circ.
\end{equation}

Since
\[
-\Delta u=\mu u,
\qquad
-\Delta V_A=\mu V_A,
\]
Green's second identity gives
\[
0
=
\int_{\partial T}
\left(
u\,\partial_\nu V_A
-
V_A\,\partial_\nu u
\right)\,ds.
\]
Using the Neumann condition for $u$ and the vanishing of
$\partial_\nu V_A$ on $AB\cup AC$, we obtain
\[
\int_{BC}u\,\partial_\nu V_A\,ds=0.
\]
Because the weight is strictly negative on $BC^\circ$, the function
$u$ must change sign on $BC^\circ$.  Indeed, if it had one sign
there, the integral could vanish only if $u$ vanished identically
on an open segment of $BC$.  Together with $\partial_\nu u=0$, this
would give vanishing Cauchy data and hence $u\equiv0$, a
contradiction.

Repeating the same argument with radial comparison functions centered
at $B$ and $C$, we conclude that $u$ changes sign on all three open
sides.  By
\cref{lem:space-form-three-side-sign}, $u$ has at least three nodal
domains, contradicting
\cref{lem:space-form-first-Neumann-two-domains}.  Hence
\[
\mu_1(T)>\nu_K(D).
\]

Now assume
\[
K>0,
\qquad
D=\frac{\pi}{2\sqrt K}.
\]
If
\[
\mu<\nu_K(D),
\]
then
\cref{lem:space-form-radial-monotonicity} gives
\[
J_\mu'(r)<0
\qquad
0<r\leq D.
\]
Thus the same Green identity argument applies even when points on
the opposite side occur at distance exactly $D$.  Therefore
\[
\mu_1(T)\geq\nu_K(D).
\]

It remains to classify the equality case.  Suppose
\[
\mu_1(T)=\nu_K(D).
\]
If $T$ is not birectangular, then for every vertex $A$ it is not the
case that both adjacent sides have length $D$.  By
\cref{lem:space-form-opposite-side-distance},
\[
d(A,X)<D
\]
on the interior of the opposite side.  Since
\[
J_{\nu_K(D)}'(r)<0
\qquad
0<r<D,
\]
the same Green identity argument again forces $u$ to change sign on
all three sides, a contradiction.  Therefore some vertex, say $A$,
satisfies
\[
AB=AC=D.
\]
Thus $T$ is birectangular.

Conversely, if
\[
AB=AC=D,
\]
then the opposite side lies in the geodesic circle
\[
d(A,\cdot)=D.
\]
The radial eigenfunction
\[
V_A(X)=J_{\nu_K(D)}(d(A,X))
\]
has vanishing normal derivative on $AB$ and $AC$ because it is
radial, and on $BC$ because
\[
J_{\nu_K(D)}'(D)=0.
\]
Hence $\nu_K(D)$ is a nonzero Neumann eigenvalue of $T$.  The lower
bound already proved yields
\[
\mu_1(T)=\nu_K(D).
\]

Finally, after rescaling the unit-sphere formula,
\[
J(r)
=
\frac{3\cos^2(\sqrt K\,r)-1}{2}
\]
satisfies
\[
-\Delta J=6KJ
\]
and
\[
J'(0)
=
J'\left(\frac{\pi}{2\sqrt K}\right)
=
0.
\]
It has exactly one zero in the interval, and hence
\[
\nu_K\left(\frac{\pi}{2\sqrt K}\right)=6K.
\]
\subsubsection{Sharpness}
It remains to prove sharpness for the non-birectangular cases.  Let
$T_{D,\alpha}=\triangle ABC$ be the isosceles geodesic triangle
satisfying
\[
AB=AC=D,
\qquad
\angle BAC=\alpha.
\]
For sufficiently small $\alpha>0$, its base has length strictly less
than $D$, and hence
\[
\diam(T_{D,\alpha})=D.
\]

Use geodesic polar coordinates centered at $A$, with the two equal
sides given by
\[
\theta=\pm\frac{\alpha}{2}.
\]
The opposite side has radial equation
\[
r=R_\alpha(\theta),
\]
where
\begin{equation}
\label{eq:space-form-thin-radial-graph}
\ct_K(R_\alpha(\theta))
=
\ct_K(D)
\frac{\cos\theta}{\cos(\alpha/2)},
\end{equation}
with
\[
\ct_K(r):=\frac{\cs_K(r)}{\sn_K(r)}.
\]
For $K=0$, this is interpreted as
\[
\frac{1}{R_\alpha(\theta)}
=
\frac{1}{D}
\frac{\cos\theta}{\cos(\alpha/2)}.
\]
It follows that
\[
R_\alpha(\theta)\longrightarrow D
\]
uniformly for
\[
|\theta|\leq\frac{\alpha}{2}
\]
as $\alpha\downarrow0$.

Let
\[
J=J_{\nu_K(D)}.
\]
Integrating \cref{eq:space-form-radial-Neumann} and using
$J'(0)=J'(D)=0$, we obtain
\begin{equation}
\label{eq:space-form-radial-mean-zero}
\int_0^D J(r)\sn_K(r)\,dr=0.
\end{equation}

Define
\[
c_\alpha
=
\frac{1}{|T_{D,\alpha}|}
\int_{T_{D,\alpha}}J(r)\,dA
\]
and
\[
f_\alpha(r,\theta)
=
J(r)-c_\alpha.
\]
Then
\[
\int_{T_{D,\alpha}}f_\alpha\,dA=0.
\]
Since
\[
dA=\sn_K(r)\,dr\,d\theta,
\]
the uniform convergence $R_\alpha\to D$ and
\cref{eq:space-form-radial-mean-zero} imply
\[
c_\alpha\longrightarrow0.
\]
Moreover,
\[
\frac1\alpha
\int_{T_{D,\alpha}}
|\nabla f_\alpha|^2\,dA
\longrightarrow
\int_0^D |J'(r)|^2\sn_K(r)\,dr,
\]
and
\[
\frac1\alpha
\int_{T_{D,\alpha}}
f_\alpha^2\,dA
\longrightarrow
\int_0^D J(r)^2\sn_K(r)\,dr.
\]
Therefore
\[
\begin{aligned}
\limsup_{\alpha\downarrow0}\mu_1(T_{D,\alpha})
&\leq
\frac{
\displaystyle\int_0^D |J'(r)|^2\sn_K(r)\,dr
}{
\displaystyle\int_0^D J(r)^2\sn_K(r)\,dr
}\\
&=
\nu_K(D).
\end{aligned}
\]
Since the strict lower bound gives
\[
\mu_1(T_{D,\alpha})>\nu_K(D)
\]
for every nondegenerate member of the family, we conclude that
\[
\mu_1(T_{D,\alpha})
\longrightarrow
\nu_K(D).
\]
Hence
\[
\inf_{\diam(T)=D}\mu_1(T)=\nu_K(D).
\]
This completes the proof.

\begin{remark}[The Euclidean case]
\label{rem:Euclidean-specialization}
When $K=0$, the radial equation \eqref{eq:space-form-Jlambda} is
\[
J_\lambda''+\frac1rJ_\lambda'+\lambda J_\lambda=0,  \qquad J_\lambda(0)=1,
\qquad
J_\lambda'(0)=0.
\]
so   
\[
J_\lambda(r)=J_0(\sqrt\lambda\,r).
\]
Since Bessel functions have the relation $J'_0(x) = J_1(x)$, the Neumann condition at $r=D$ is $
J_1(\sqrt\lambda\,D)=0. $
Consequently,
\[
\nu_0(D)
=
\frac{j_{1,1}^2}{D^2},
\]
where $j_{1,1}$ is the first positive zero of $J_1$.  Thus \cref{The Main Theorem} recovers the sharp Euclidean estimate
\[
\mu_1(T)>
\frac{j_{1,1}^2}{D^2}
\]
for every nondegenerate Euclidean triangle as proved in \cite{LaugesenSiudeja2010}.
\end{remark}

\subsection{Properties of the model}
\subsubsection{Comparison of the models}
Let $\bar\mu_K(D)$ denote the sharp one-dimensional comparison
constant for general convex domains in a two-dimensional space form.
Thus $\bar\mu_K(D)$ is the first positive Neumann eigenvalue of
\begin{equation}
\label{eq:convex-model-comparison}
-\bigl(\cs_K(s)\Phi'(s)\bigr)'
=
\bar\mu_K(D)\cs_K(s)\Phi(s),
\qquad
-\frac D2<s<\frac D2,
\end{equation}
with
\[
\Phi'\left(-\frac D2\right)
=
\Phi'\left(\frac D2\right)
=
0.
\]
Recall $\nu_K(D)$ is the first positive
eigenvalue of
\begin{equation}
\label{eq:triangle-model-comparison}
-\bigl(\sn_K(r)J'(r)\bigr)'
=
\nu_K(D)\sn_K(r)J(r),
\qquad
0<r<D,
\end{equation}
with
\[
J'(0)=J'(D)=0.
\]

From the definition and the equality cases of the models, we have the following. 
\begin{proposition}
\label{prop:triangle-model-strictly-better}
Let \(K\in\mathbb R\). Assume \(D>0\) if \(K\leq 0\), and $0<D\leq \frac{\pi}{2\sqrt K}$
if \(K>0\). Then
\[
\nu_K(D)>\overline{\mu}_K(D).
\]
\end{proposition}
For $K \le 0$, this follows from \cite[Theorem 13]{BakryQian2000}. Below we give a precise proof for all $K$. 
\begin{proof}

We first give general formula for transform weighted Neumann problems into Dirichlet
Schrödinger problem. Let \(w>0\) in the interior of an interval and
suppose that $
-(wu')'=\lambda wu.$
Set
\[
a=\frac{w'}{w},
\qquad
v=\sqrt{w}\,u'.
\]
A direct computation gives
\begin{equation}
-v''
+
\left(
\frac{a^2}{4}-\frac{a'}{2}
\right)v
=
\lambda v.
\label{eq:weighted-Neumann-to-Schrodinger}
\end{equation}
If \(u\) is the first positive Neumann eigenfunction, then \(u'\) has
one strict sign in the interior. Thus \(v\), up to an overall sign, is
positive in the interior and vanishes at the endpoints. It is
therefore the first Dirichlet eigenfunction of the operator in
\cref{eq:weighted-Neumann-to-Schrodinger}.

For the triangle model \cref{eq:triangle-model-comparison}, $
w(r)=\sn_K(r), \ 0<r<D.$
We obtain
\begin{equation}
\nu_K(D)
=
\lambda_1^D\bigl((0,D);V_{\mathrm r}\bigr),
\qquad
V_{\mathrm r}(r)
=
\frac{3}{4\sn_K^2(r)}-\frac K4.
\label{eq:radial-Schrodinger-potential}
\end{equation}

For the convex model \cref{eq:convex-model-comparison}, $
w_{\mathrm c}(s)=\cs_K(s), \ 
-\frac D2<s<\frac D2.$
We obtain
\begin{equation}
\overline{\mu}_K(D)
=
\lambda_1^D
\left(
\left(-\frac D2,\frac D2\right);
V_{\mathrm c}
\right),
\qquad
V_{\mathrm c}(s)
=
\frac{3K}{4\cs_K^2(s)}-\frac K4.
\label{eq:centered-Schrodinger-potential}
\end{equation}

We now compare the two Dirichlet problems.
\medskip

\noindent
\emph{Case 1: \(K\leq 0\).} 
By \cref{eq:radial-Schrodinger-potential} and \cref{eq:centered-Schrodinger-potential}
\[
V_{\mathrm r}(r)>-\frac K4 \ge
V_{\mathrm c}(s). 
\]
Since both intervals have length \(D\), the variational
characterization and strict monotonicity with respect to the potential
give
\[
\nu_K(D)
>
\frac{\pi^2}{D^2}-\frac K4
\geq
\overline{\mu}_K(D).
\]

\medskip
\noindent
\emph{Case 2: \(K>0\).}
Set $R_K:=\frac{\pi}{2\sqrt K}$
and define
\begin{equation}\label{eq: Positive potential V_K}
    V_K(x)
=
\frac{3K}{4\cos^2(\sqrt K x)}-\frac K4,
\qquad
-R_K<x<R_K.
\end{equation}
Under the translation $
x=r-R_K,$
we have
\[
\cos(\sqrt Kx)
=
\cos\left(\sqrt Kr-\frac{\pi}{2}\right)
=
\sin(\sqrt Kr).
\]
Consequently,
\[
\nu_K(D)
=
\lambda_1^D
\left(
(-R_K,-R_K+D);V_K
\right).
\]
Similarly,
\[
\overline{\mu}_K(D)
=
\lambda_1^D(I_0;V),
\qquad
I_0=\left(-\frac D2,\frac D2\right).
\]

Let \(\phi>0\) be an \(L^2\)-normalized first Dirichlet eigenfunction
on \(I_{\mathrm r}\), extended by zero outside \(I_{\mathrm r}\), and
let \(\phi^*\) be its symmetric decreasing rearrangement. Since
\(\lvert I_{\mathrm r}\rvert=D\), one has
$
\int_{I_0}(\phi^*)^2\,dx
=
\int_{I_{\mathrm r}}\phi^2\,dx,
$
and
$
\int_{I_0}|(\phi^*)'|^2\,dx
\leq
\int_{I_{\mathrm r}}|\phi'|^2\,dx.
$
Because \(V_K\) is even and strictly increasing in \(\lvert x\rvert\),
the rearrangement inequality gives
\[
\int_{I_0}V_K(x)(\phi^*(x))^2\,dx
<
\int_{I_{\mathrm r}}V_K(x)\phi(x)^2\,dx.
\]
The inequality is strict because \(\phi>0\) throughout
\(I_{\mathrm r}\), while the interval \(I_{\mathrm r}\) is not
centered at the origin.

Consequently,
\[
\begin{aligned}
\overline{\mu}_K(D)
&\leq
\frac{
\displaystyle
\int_{I_0}
\left(
|(\phi^*)'|^2+V_K(\phi^*)^2
\right)\,dx
}{
\displaystyle
\int_{I_0}(\phi^*)^2\,dx
}
<
\frac{
\displaystyle
\int_{I_{\mathrm r}}
\left(
|\phi'|^2+V\phi^2
\right)\,dx
}{
\displaystyle
\int_{I_{\mathrm r}}\phi^2\,dx
}
=
\nu_K(D).
\end{aligned}
\]

The two cases together prove
\[
\nu_K(D)>\overline{\mu}_K(D).
\]
\end{proof}

\subsubsection{Monotonicity of $\nu_K(D)$}
\begin{lemma}
\label{lem:nu-monotone} The first positive Neumann eigenvalue $\nu_K(D)$ of the model \cref{eq:space-form-radial-Neumann} is strictly decreasing on $(0,\infty)$ if $K\leq 0$, and on $\left(0,\frac{\pi}{\sqrt K}\right)$ if $K>0.$
More precisely, if $J_D$ is a corresponding first positive radial
Neumann eigenfunction, then
\begin{equation}
\label{eq:nu-derivative}
\nu_K'(D)
=
-\nu_K(D)\,
\frac{
\sn_K(D)\,J_D(D)^2
}{
\displaystyle
\int_0^D J_D(r)^2\sn_K(r)\,dr
}
<0.
\end{equation}
\end{lemma}

\begin{proof}
Write
\[
\nu=\nu_K(D),
\qquad
J=J_D,
\qquad
w(r)=\sn_K(r).
\]
Since the first positive radial Neumann eigenvalue is simple, we may
choose $\nu_K(D)$ and $J_D$ differentiably with respect to $D$.
Let a dot denote differentiation with respect to $D$, with $r$ fixed.

Differentiating
\[
-(wJ')'=\nu wJ
\]
with respect to $D$ gives
\[
-(w\dot J')'
=
\nu' wJ+\nu w\dot J.
\]
Multiplying by $J$ and integrating over $(0,D)$ yields
\[
-\bigl[Jw\dot J'\bigr]_0^D
+
\int_0^D wJ'\dot J'\,dr
=
\nu'\int_0^D wJ^2\,dr
+
\nu\int_0^D wJ\dot J\,dr.
\]
On the other hand, multiplying the original eigenvalue equation by
$\dot J$ and integrating by parts, using
\[
J'(0)=J'(D)=0,
\]
gives
\[
\int_0^D wJ'\dot J'\,dr
=
\nu\int_0^D wJ\dot J\,dr.
\]
Subtracting the two identities, we obtain
\begin{equation}
\label{eq:nu-derivative-intermediate}
-\bigl[Jw\dot J'\bigr]_0^D
=
\nu'\int_0^D wJ^2\,dr.
\end{equation}

Since $w(0)=0$, the contribution at $r=0$ vanishes. At the moving
endpoint, differentiate the Neumann condition
\[
J_D'(D)=0
\]
to obtain
\[
\dot J'(D)+J''(D)=0.
\]
Since $J'(D)=0$, the eigenvalue equation at $r=D$ gives
\[
J''(D)=-\nu J(D),
\]
and hence
\[
\dot J'(D)=\nu J(D).
\]
Substituting this into
\cref{eq:nu-derivative-intermediate} gives
\[
-\nu w(D)J(D)^2
=
\nu'\int_0^D wJ^2\,dr.
\]
Therefore
\[
\nu_K'(D)
=
-\nu_K(D)\,
\frac{
\sn_K(D)\,J_D(D)^2
}{
\displaystyle
\int_0^D J_D(r)^2\sn_K(r)\,dr
}.
\]

For the stated ranges of $D$,
\[
\sn_K(D)>0.
\]
Also $J_D(D)\neq0$, since otherwise
\[
J_D(D)=J_D'(D)=0
\]
would imply $J_D\equiv0$ by uniqueness for the ODE. Thus every factor
on the right-hand side except the leading minus sign is positive, and
hence
\[
\nu_K'(D)<0.
\]
\end{proof}

\subsection{Explicit lower bound}
\subsubsection{For positively curved case}

Let $K> 0$. Recall the Schr\"odinger potential \cref{eq: Positive potential V_K}
\begin{equation*}
    -v''+\left(\frac{3K}{4\sin^2(\sqrt{K}r)}-\frac{K}{4}\right)v=\nu_K(D)v
\end{equation*}
with Dirichlet conditions $v(0)=v(D)=0$, and the Rayleiqh quotient characterization
\begin{align*}
    \nu_K(D)=\inf_{v\in H^{1}_0(0,D)\backslash \{0\}}\frac{\int_0^D(|v'|^2+V_K(r)v^2)dr}{\int_0^Dv^2 dr}.
\end{align*}

The eigenvalue problem is exactly solvable for $D_0=\frac{\pi}{2\sqrt{K}}$.  This corresponds geometrically to the hemisphere case.  For $0<D\leq D_0$, define $\delta:=\frac{D}{D_0}=\frac{2\sqrt{K}D}{\pi}$.

In order to apply the exact bounds, we rescale the problem from $(0,D)$ to $(0,D_0)$.  Let
\begin{align*}
    w(x):=\sqrt{\delta}v(\delta x)
\end{align*}
so that
\begin{align*}
    \int_0^{D_0}w^2dx &= \int_0^D v^2dr\\
    \int_0^{D_0}|w'|^2dx&=\delta^2\int_0^D|v'|^2dr\\
    \int_0^{D_0}V_K(\delta x)w(x)^2dx&=\int_0^DV_K(r)v(r)^2dr.
\end{align*}
Applying these to the Rayleigh quotient we get
\begin{align*}
    \nu(D)=\inf_{w\in H^{1}_0(0,D_0)\backslash \{0\}}\frac{\delta^{-2}\int_0^{D_0}|w'|^2dx+\int_0^{D_0}V_K(\delta x)w^2dx}{\int_0^{D_0}w^2 dx}.
\end{align*}
At $D_0$, the eigenvalue is 6K hence
\begin{align*}
    \int_0^{D_0}(|w'|^2+V_K(x)w^2)dx \geq 6K\int_0^{D_0}w^2.
\end{align*}
Wirtinger's inequality on $(0,D_0)$ gives
\begin{align*}
    \int_0^{D_0}|w'|^2dx \geq \frac{\pi^2}{D_0^2}\int_0^{D_0}w^2dx=4K\int_0^{D_0}w^2dx.
\end{align*}
By direct computation, we have 
\begin{align*}
    V_K(\delta x)- V_K(x)=\frac{3K}{4}(\csc^2(\delta\sqrt{K}x)-\csc^2(\sqrt{K}x).
\end{align*}
This is decreasing in $x$ for $(0,\frac{\pi}{2\sqrt{K}}]$.  Hence the minimum occurs at $x=D_0$.
Now
\begin{align*}
    \delta^{-2}&\int_0^{D_0}|w'|^2dx+\int_0^{D_0}V_K(\delta x)w^2dx\\ &= \int_0^{D_0}(|w'|^2+V_K(x)w^2)dx+(\delta^{-2}-1)\int_0^{D_0}|w'|^2dx+\int_0^{D_0}(V_K(\delta x)-V_K(x))w^2dx\\
    &\geq 6K+4K(\delta^{-2}-1)+\frac{3}{4}\left(\frac{1}{\sn_K(D)^2}-K\right).
\end{align*}

Hence we have
\begin{align}  \label{eq: nu explicite lower bound}
    \nu_K(D)\geq \frac{\pi^2}{D^2}+\frac{5K}{4}+\frac{3K}{4\sin^2(\sqrt{K}D)} \ge 6K.
\end{align}

\subsubsection{For negatively curved case}
Let $K<0$ and set $K=-\kappa^2$, $\kappa >0$ using $v(r)=-\sqrt{\frac{\sinh(\kappa r)}{\kappa}}J'(r)$, the Schr\"odinger normal form is
\begin{align*}
    -v''+\left(\frac{\kappa^2}{4}+\frac{3\kappa^2}{4\sinh^2(\kappa r)}\right)v=\nu_K(D)v,
\end{align*}
with Dirichlet conditions $v(0)=v(D)=0$.  Let $V_K(r)=\frac{\kappa^2}{4}+\frac{3\kappa^2}{4\sinh^2(\kappa r)}$.  By direct computation, $V_K(r)$ is decreasing so that $V_K(r)\geq V_K(D)$.  Combining with Wirtinger's inequality on $[0,D]$, we get
\begin{align*}
    \nu_K(D)\geq \frac{\pi^2}{D^2}+\frac{\kappa^2}{4}+\frac{3\kappa^2}{4\sinh^2(\kappa D)}
\end{align*}
for any $D>0$.  Unlike in the $K>0$ case, we do not have a finite reference interval $[0,D_0]$ which we can compare our problem to on $[0,D]$ since the one-dimensional model does not admit any elementary polynomial solutions.  However, we can bound the potential using $\csch^2(t)\geq \frac{1}{t^2}-\frac{1}{3}$ so that 
\begin{align*}
    V_K(r)&=\frac{\kappa^2}{4}+\frac{3\kappa^2}{4\sinh^2(\kappa r)}\\
    &\geq \frac{\kappa^2}{4}+\frac{3}{4}\left(\frac{1}{r^2}-\frac{\kappa^2}{3}\right)=\frac{3}{4r^2}.
\end{align*}
Hence the comparison operator is $-\frac{d^2}{dr^2}+\frac{3}{4r^2}$ which has the first Dirichlet eigenvalue $\frac{j_{1,1}^2}{D^2}$ hence we obtain a curvature independent estimate $\nu_K(D)\geq\frac{j_{1,1}^2}{D^2}$.  Combining the two estimates we get
\begin{equation}\label{nu lower bound K<0}
    \nu_K(D) \geq \max\left\{\frac{j^2_{1,1}}{D^2},\frac{\pi^2}{D^2}+\frac{|K|}{4}+\frac{3|K|}{4\sinh^2(\sqrt{|K|}D)}\right\}.
\end{equation}

\section{Hot spots for non acute spherical triangles}\label{section:Spherical Hatcher Theorem}
In this section, we adapt most lemmas and arguments from \cite{Hatcher2025}. To prove \cref{sphericalHatcherTheorem}, according to \cite[Theorem~1.5]{Hatcher2025} and the discussion therein, we need to show the key proposition
\begin{proposition}[Spherical mixed-Neumann eigenvalue comparison]\label{Spherical Hatcher lemma 5.2}
    Let $T \subset \mathbb{S}^2$ be a geodesic triangle contained in the upper hemisphere with diameter $D\leq \pi/2$. After applying an isometry, assume that the distinguished side $e$ lies on the equator. Let $\mu=\mu_1(T)$ be the first positive Neumann eigenvalue of $T$, and let $\lambda_1(T ; e)$ be the first mixed eigenvalue with Neumann condition on $e$ and Dirichlet condition on the other two sides. Then
    $$
\mu_1^N(T) \leq \lambda_1(T ; e) .
$$
\end{proposition}
This proposition is the exact analogue of Hatcher's Lemma 5.2 in \cite{Hatcher2025}. For our purpose, it is sufficient to show it holds for triangles with diameter $D\leq \pi/2.$ We now prove it.

Without loss of generality, we can assume $T$ is contained in the first octant, as its diameter is at most $\pi/2$. Let $t$ be the distance from the equator, measured into the upper hemisphere, and let $s$ be arclength along the equator. Then

$$
g=d t^2+\cos ^2 t d s^2, \quad d A=\cos t d s d t
$$

and the Laplacian is

$$
\Delta=\partial_{t t}-\tan t \partial_t+\sec ^2 t \partial_{s s} .
$$

The side $e$ is contained in $\{t=0\}$. For each $t$, consider the latitude slice

$$
T_t=T \cap\{d(\cdot, e)=t\}
$$

Assume this slice is represented by an interval $s_{-}(t)<s<s_{+}(t).$ Define

$$
L(t)=s_{+}(t)-s_{-}(t) .
$$

Thus $L(t)$ is the angular width of $T$ at height $t$. It is also the length of the orthogonal projection of the latitude slice onto the equator. The actual spherical length of the slice is $\ell(t)=\cos t L(t).$

Let $h$ denote the maximal distance of points of $T$ from the equator. We now follow the ideas of the proof of \cite[Lemma~5.2]{Hatcher2025} and \cite[Theorem~3.1]{Rohleder2017} and construct a test function.
\subsection{Construction of the test function}
For notational purposes, denote $\mu=\mu_1^N(T)$ the first positive Neumann eigenvalue of $T$ and $\lambda_1=\lambda_1(T ; e)$ the first mixed eigenvalue. Choose $a>0$ so that
$$
a(a+1)=\mu.
$$
Equivalently,
$$
a=\frac{-1+\sqrt{1+4 \mu}}{2}.
$$
Consider the complex-valued function
\begin{equation}\label{eq:w}
w(t, s)=(\cos t)^a e^{i a s}.
\end{equation}
Its derivatives are
$$
w_t=-a \tan t w, \quad w_s=i a w.
$$

A direct calculation gives
$$
-\Delta w=a(a+1) w=\mu w.
$$
Also,
$$
w_t(0, s)=0,
$$
so $w$ satisfies the Neumann condition on the equatorial side $e$. This is the spherical analogue of Hatcher's hyperbolic exponential $y^s$.
Define
$$
B_\mu(f)=\int_T|\nabla f|^2 d A-\mu \int_T|f|^2 d A.
$$

For the function $w$,
$$
|\nabla w|^2=\left|w_t\right|^2+\sec ^2 t\left|w_s\right|^2 .
$$
Therefore,
$$
|\nabla w|^2=a^2\left(\tan ^2 t+\sec ^2 t\right)(\cos t)^{2 a}.
$$

Using $\mu=a(a+1)$, we obtain
$$
|\nabla w|^2-\mu|w|^2=a(\cos t)^{2 a}\left(2 a \tan ^2 t-1\right).
$$
Since this expression is independent of $s$, integration over each latitude slice gives
$$
B_\mu(w)=a \int_0^h L(t) \cos ^{2 a+1} t\left(2 a \tan ^2 t-1\right) d t.
$$

Now define
$$
F(t)=\sin t \cos ^{2 a} t .
$$
Then
$$
F^{\prime}(t)=\cos ^{2 a+1} t\left(1-2 a \tan ^2 t\right).
$$
Consequently,
$$
B_\mu(w)=-a \int_0^h L(t) F^{\prime}(t) d t.
$$

Integrating by parts gives
$$
B_\mu(w)=-a[L(t) F(t)]_0^h+a \int_0^h L^{\prime}(t) F(t) d t.
$$

The boundary term vanishes:
\begin{itemize}
    \item At $t=0, F(0)=0$.
    \item At the top of the triangle, either $L(h)=0$, because the slice collapses to the opposite vertex, or $h= \pi / 2$, in which case $F(h)=0$.
\end{itemize}
Hence
$$
B_\mu(w)=a \int_0^h L^{\prime}(t) \sin t \cos ^{2 a} t d t.
$$
This is the main identity. Because $a>0$ and
$$
\sin t \cos ^{2 a} t>0
$$
for $0<t<h<\pi / 2$, we immediately obtain:
If $L^{\prime}(t) \leq 0$ almost everywhere and $L^{\prime}(t)<0$ on a set of positive measure, then
$$
B_\mu(w)<0 .
$$
\begin{remark}
    More generally, pointwise monotonicity is not necessary. It is enough that
$$
\int_0^h L^{\prime}(t) \sin t \cos ^{2 a} t d t<0
$$
Thus some positive values of $L^{\prime}$ are allowed, provided that the weighted negative part dominates.
\end{remark}

\subsection{The derivative of the angular width $L$}
\begin{lemma}
      $L'(t)\leq 0$ is true for triangles $T :=\triangle PQR \subset$ first octant. Moreover, the inequality is strict unless
$$
d(P, R)=d(Q, R)=\frac{\pi}{2} .
$$
In case of equality, $R=O$ is the pole of the great circle containing $P Q$; the triangle $\triangle PQR$ is a bi-rectangular isosceles triangle; and $L(t) \equiv d(P, Q) .$ 
In particular, if $\bar{T}$ is contained in the interior of the first octant, then $L^{\prime}(t)<0$ for every side $e$ of $T$ and every interior latitude $0<t<h$.
\end{lemma}

\begin{proof}
    Fix a triangle $\triangle PQR$ in the first octant of the sphere. Fix $e=PQ$ to be the Neumann edge. All three side lengths of $T$ are at most $\pi / 2$. Write
    $$
a=d(Q, R), \quad b=d(P, R), \quad c=d(P, Q)
$$
so that
$$
a, b, c \leq \frac{\pi}{2}.
$$

Let $\alpha$ and $\beta$ be the interior angles at $P$ and $Q$, respectively.
\par Step 1: Formula for the latitude width $L$:
After applying an isometry, suppose that $P Q$ lies on the equator and use coordinates
$$
X(t, s)=(\cos t \cos s, \cos t \sin s, \sin t),
$$
with
$$
P=X(0,0), \quad Q=X(0, c).
$$

At $P$, the unit tangent vector to $P R$, directed toward $R$, is
$$
v_P=\cos \alpha \partial_s+\sin \alpha \partial_t .
$$

The great circle containing $P R$ is the intersection of $\mathbb{S}^2$ with the plane spanned by $P$ and $v_P$. Computing a normal to this plane gives the equation
$$
\sin s=\cot \alpha \tan t .
$$

Because $b=d(P, R) \leq \pi / 2$, latitude increases monotonically along $P R$, so this side remains on the principal inverse-sine branch. Its longitude is therefore
$$
s_{-}(t)=\arcsin (\cot \alpha \tan t) .
$$

Similarly, the side $Q R$ satisfies
$$
\sin (c-s)=\cot \beta \tan t,
$$
and, since $a=d(Q, R) \leq \pi / 2$, it also remains on its principal branch:
$$
s_{+}(t)=c-\arcsin (\cot \beta \tan t) .
$$

Consequently,
$$
L(t)=c-\arcsin (\cot \alpha \tan t)-\arcsin (\cot \beta \tan t) .
$$

Differentiating gives
$$
L^{\prime}(t)=-\sec ^2 t\left[\frac{\cot \alpha}{\sqrt{1-\cot ^2 \alpha \tan ^2 t}}+\frac{\cot \beta}{\sqrt{1-\cot ^2 \beta \tan ^2 t}}\right] .
$$
It remains to prove that the expression in brackets is nonnegative.
\par Step 2: A relation between the inner angles $\alpha$ and $\beta:$
Let $\gamma$ be the angle of $T$ at $R$. The spherical cosine and sine laws give
$$
\cot \alpha=\frac{\cos a-\cos b \cos c}{\sin b \sin c \sin \alpha}=\frac{\cos a-\cos b \cos c}{\sin a \sin b \sin \gamma} .
$$

Similarly,
$$
\cot \beta=\frac{\cos b-\cos a \cos c}{\sin a \sin b \sin \gamma} .
$$
Adding these two identities yields
$$
\cot \alpha+\cot \beta=\frac{(\cos a+\cos b)(1-\cos c)}{\sin a \sin b \sin \gamma} .
$$

The denominator is positive because the triangle is nondegenerate. Also, $1-\cos c>0 .$

Since $a, b \leq \pi / 2$, $\cos a \geq 0, \quad \cos b \geq 0 .$ Therefore,
$$
\cot \alpha+\cot \beta \geq 0 .
$$

The inequality is strict unless
$$
\cos a=\cos b=0,
$$
which is equivalent to $a=b=\frac{\pi}{2} .$

\par Step 3: sign of $L'(t):$ Fix $t \in(0, h)$, set $q=\tan t$, and define
$$
F_q(x)=\frac{x}{\sqrt{1-q^2 x^2}}
$$

On the relevant interval, $F_q$ is odd and strictly increasing. From $\cot{\alpha}+\cot{\beta} \geq 0$ we obtain $\cot{\beta} \geq-\cot{\alpha}.$ Because $F_q$ is increasing and odd,
$$
F_q(\cot{\beta}) \geq F_q(-\cot{\alpha})=-F_q(\cot{\alpha}) \text {. }
$$
Hence
$$
F_q(\cot{\alpha})+F_q(\cot{\beta}) \geq 0
$$
The derivative formula now gives $L^{\prime}(t) \leq 0$ as desired.

If $a$ and $b$ are not both $\pi / 2$, then
$$
\cot{\alpha}+\cot{\beta}>0
$$
so $\cot{\beta}>-\cot{\alpha}$. Strict monotonicity of $F_q$ then gives
$$
F_q(\cot{\alpha})+F_q(\cot{\beta})>0
$$
and therefore $L^{\prime}(t)<0.$
\par Step 4: The equality case:
Suppose
$$
a=b=\frac{\pi}{2}
$$

The spherical cosine law gives
$$
\cos a=\cos b \cos c+\sin b \sin c \cos \alpha.
$$
Since $a=b=\pi / 2$, this becomes
$$
0=\sin c \cos \alpha.
$$

Because $c>0$, we have $\sin c>0$, and hence
$$
\alpha=\frac{\pi}{2}.
$$

Similarly,
$$
\beta=\frac{\pi}{2}.
$$
Thus both $P R$ and $Q R$ are perpendicular to the great circle containing $P Q$. Their common endpoint $R$ is its pole. The longitude boundaries are constant:
$$
s_{-}(t)=0, \quad s_{+}(t)=c.
$$
Therefore,
$$
L(t) \equiv c, \quad L^{\prime}(t) \equiv 0.
$$

For the equality case, one can manually compute $\lambda_1(T;e)= \frac{\pi}{\beta}(\frac{\pi}{\beta}+1) \geq 6 = \mu_1^N(T)$ using separation of variables method and conclude $\mu_1^N(T) \leq \lambda_1(T;e).$

Finally, suppose that $\bar{T}$ lies in the interior of $T_{\pi/2}$. All coordinates of all points of $\bar{T}$ are then strictly positive. Hence, for distinct $X, Y \in \bar{T}$,
$$
X \cdot Y>0
$$
and therefore
$$
d(X, Y)<\frac{\pi}{2}.
$$
Thus the exceptional case cannot occur for any choice of edge, and
$$
L^{\prime}(t)<0
$$
for every side $e$ and every $0<t<h$.
\end{proof}

\subsubsection{Proof of Proposition \ref{Spherical Hatcher lemma 5.2}}
Assume for a contradiction that $\lambda_1(T;e)<\mu_1^N(T)$. Let $\varphi$ be the first mixed eigenfunction and $w$ defined in \eqref{eq:w}. Define the Hermitian bilinear form
$$
B_\mu(f, g)=\int_T\langle\nabla f, \nabla \bar{g}\rangle d A-\mu \int_T f \bar{g} d A.
$$

Then
$$
B_\mu(\varphi, \varphi)=\left(\lambda_1-\mu\right) \int_T \varphi^2 d A<0
$$
and our calculation above gives
$$
B_\mu(w, w)=B_\mu(w)<0 .
$$

The cross term vanishes:
$$
\begin{aligned}
B_\mu(\varphi, w) & =\int_T\langle\nabla \varphi, \nabla \bar{w}\rangle-\mu \int_T \varphi \bar{w} \\
& =\int_{\partial T} \varphi \partial_\nu \bar{w} d s =0.
\end{aligned}
$$
because $\varphi=0$ on $\partial T \backslash e$, while $\partial_\nu w=0$ on $e$. Therefore
$$
B_\mu(a \varphi+b w, a \varphi+b w)<0
$$
for every nonzero $(a, b) \in \mathbb{C}^2$. Since $\varphi$ and $w$ are linearly independent, their span is a two-dimensional trial space on which every Rayleigh quotient is strictly below $\mu$, contradicting the min-max characterization of the first positive Neumann eigenvalue. 
\qed

\begin{remark}
    As commented in \cite[Example~3.6]{Rohleder2017}, in general the inequality in \cref{Spherical Hatcher lemma 5.2} is not strict.
\end{remark}
Hatcher explicitly claims that if this lemma is proven, then all of his arguments afterwards extend to the positive curvature case. In particular,

\begin{proposition}[Spherical analogue of Lemma 7.1 \cite{Hatcher2025}]\label{sphericalLemma7.1}
    Let $T\subset \Sph$ be a geodesic triangle of diameter $D \leq \pi/2.$ Let $D, N \subset \partial T$ be unions of edges forming a partition of $\partial T$, and let $e \subset N$ be the closure of one Neumann edge. Let $X$ be a spherical Killing field tangent to the great circle containing $e$.
\begin{itemize}
\item[1.] Suppose $D \neq \varnothing$, and let $v$ be a first mixed eigenfunction with eigenvalue $\lambda_1^D(T)$. Set
$$
\phi=X v .
$$
Then $Z(\phi)$ cannot contain either a loop or an arc with both endpoints in $e$, unless
$$
\phi \equiv 0 \quad \text { on } \partial T \backslash e .
$$

\item[2.] Suppose $D=\varnothing$, and let $u$ be a first nonconstant Neumann eigenfunction with eigenvalue $\mu_1(T)$. Let
$$
\phi=u \quad \text { or } \quad \phi=X u .
$$
Then the same conclusion holds.
\end{itemize}
\end{proposition}
The proof is essentially identical to Hatcher's Lemma 7.1, and we just use \cref{Spherical Hatcher lemma 5.2} to replace his Lemma 5.2 when needed.

Consequently, all lemmas in \cite[section~7]{Hatcher2025} hold. We state an important lemma that says the multiplicity of the first nonzero Neumann eigenfunction is at most two:
\begin{lemma}[Spherical analogue of Hatcher's Corollary 7.4]
     The second Neumann eigenspace of a positive constant-curvature geodesic triangle $T$ with diameter $D\leq \pi/2$ has dimension at most two.
\end{lemma}

\subsection{A deformation argument}
In \cite[section~9]{Hatcher2025}, Hatcher constructed a path using the Klein model from the flat Euclidean space to constant curvature space $M_{\kappa}$. As we only consider triangles of diameter $D\leq \pi/2$, we need to check the deformation step stays valid in our case. We now check the deformation stays inside the $\pi/2$ equilateral triangle. 

Let $T_1 \subset \mathbb{S}^2$ be a non-acute spherical geodesic triangle on the unit sphere such that $\operatorname{diam}\left(T_1\right) \leq \frac{\pi}{2}.$ Assume first that $T_1$ is not a triangle with two right-angled vertices. Choose a non-acute vertex $v_0$, move it to the origin of the spherical Klein model, and let
$$
\Delta=\operatorname{conv}\left\{0, x_1, x_2\right\} \subset \mathbb{R}^2
$$
be the fixed Euclidean triangle representing $T_1$.
For $0 \leq \kappa \leq 1$, equip $\Delta$ with Hatcher's Klein metric
$$
g_\kappa=\frac{d r^2}{\left(1+\kappa r^2\right)^2}+\frac{r^2}{1+\kappa r^2} d \theta^2.
$$

Denote the resulting geodesic triangle by $T_\kappa$. Then:
\begin{enumerate}
    \item the angle of $T_\kappa$ at the origin is independent of $\kappa$, so $T_\kappa$ remains non-acute;
\item  for every $0<\kappa \leq 1$,
$
\sqrt{\kappa} \operatorname{diam}_{g_\kappa}\left(T_\kappa\right) \leq \frac{\pi}{2};
$

\item the inequality is strict for $0<\kappa<1$.
\end{enumerate}
At curvature $\kappa>0$, rescale the metric by
$$
\tilde{g}_\kappa=\kappa g_\kappa .
$$

Then 
\begin{lemma}
$\tilde{g}_\kappa$ has curvature 1, and
$$
\operatorname{diam}_{\tilde{g}_\kappa}\left(T_\kappa\right)=\sqrt{\kappa} \operatorname{diam}_{g_\kappa}\left(T_\kappa\right) \leq \frac{\pi}{2} .
$$    
\end{lemma}

\begin{proof}
    Let the non-acute vertex be the Klein origin, and let the other vertices be $x_1, x_2 \in \mathbb{R}^2$. For the rescaled metric
$$
\tilde{g}_\kappa=\kappa g_\kappa,
$$
the unit-sphere realization is
$$
\Phi_\kappa(x)=\frac{(1, \sqrt{\kappa} x)}{\sqrt{1+\kappa|x|^2}} .
$$

Consequently,
$$
d_{\tilde{g}_\kappa}\left(0, x_i\right)=\arctan \left(\sqrt{\kappa}\left|x_i\right|\right),
$$
which is strictly increasing in $\kappa$.
If $\gamma \geq \pi / 2$ is the angle at the origin, then $x_1 \cdot x_2 \leq 0$, and
$$
\cos d_{\tilde{g}_\kappa}\left(x_1, x_2\right)=\frac{1+\kappa x_1 \cdot x_2}{\sqrt{\left(1+\kappa\left|x_1\right|^2\right)\left(1+\kappa\left|x_2\right|^2\right)}} .
$$

The right-hand side is strictly decreasing in $\kappa$, so this side length is also strictly increasing. Thus all three side lengths at $0<\kappa<1$ are strictly smaller than their values at $\kappa=1$.

 In Klein coordinates, every point of the triangle lifts to the normalization of a positive linear combination of the three lifted vertex vectors. Since all pairwise vertex inner products are nonnegative, any two such positive combinations have nonnegative inner product. Hence every two points of the triangle are at spherical distance at most $\pi / 2$.
\end{proof}

Along the Klein path, every positive-curvature triangle, after constant rescaling to curvature 1, has diameter at most $\pi / 2$. \cref{Spherical Hatcher lemma 5.2} therefore supplies the analogue of Hatcher's Lemma 5.2 at every positive-curvature parameter. Consequently, the proof of Hatcher's Lemma 7.1 and Corollaries 7.2-7.4 applies to every triangle on the path. The spectral continuation and critical-point arguments in Hatcher's Sections 9, 10, and 12 then apply verbatim, as asserted in Hatcher's Theorem 1.5. We have therefore proved \cref{sphericalHatcherTheorem}.

\section{Antisymmetry and eigenvalue monotonicity for isosceles triangles}\label{Section: lower bound isoceles}
In this section we prove \cref{Antisymmetry of isosceles triangles Theorem} and that the first nonzero Neumann eigenvalue of an isosceles triangle of diameter $\pi/2$ is at least 6. The case of birectangular isosceles triangles can be done by the exceptional case in \cref{The Main Theorem}. We now just deal with the case in which the base realizes diameter. We restate \cref{Antisymmetry of isosceles triangles Theorem} here:
\begin{theorem}\label{thm: antisymmetry of isosceles eigenfunction}
    Assume $T$ is non-birectangular isosceles and diameter $D=\pi/2$. Then the first nonzero Neumann eigenfunction $u$ of $T$ is antisymmetric about its median.
\end{theorem}

\begin{proof}
    Let $OA$ be its base and $s=|AB|=|OB|$ be its side. Then $T$ being a non-birectangular isosceles triangle implies the base $OA$ realizes its diameter: $|OA|=D=\pi/2,$ and $D/2<s<D.$ The spherical cosine law implies $T$ must be an obtuse triangle:
    \[0=\cos D=\cos ^2 s+\sin ^2 s \cos \angle ABO\]
    and hence $\cos \angle ABO<0.$
    
    Let $M$ be the midpoint of $OA.$ Then $T$ is symmetric about its median $BM$. By \cref{sphericalHatcherTheorem} $\mu_1^N(T)$ is simple, and hence its corresponding eigenfunction $u$ is either symmetric or antisymmetric about $BM$. Let $\mu_s, u_s$ be the first nonzero symmetric Neumann eigenvalue and eigenfunction, and $\mu_a$ be the antisymmetric counterpart. We claim:
    \begin{claim}
        $\mu_a<\mu_s.$
    \end{claim}
    If not, the first Neumann eigenfunction would be $u_s$. Since $u_s$ is symmetric about the median $BM$, its normal derivative along $BM$ vanishes. In particular, at $M$, its normal derivative along $BM$ is actually the tangential derivative along the base $OA$, hence
\[\partial_\tau u_s(M)=0 .\] Also the normal derivative of $u_s$ vanishes on the side $OA.$ Hence it implies $M$ is a critical point at the interior of an edge, contradicting \cref{sphericalHatcherTheorem}. It follows immediately that $\mu=\mu_a$ and the first Neumann eigenfunction is the antisymmetric one.
\end{proof}

To prove the lower bound $\mu>6$ for such $T$, we can in fact show a stronger statement:
\begin{theorem}[Strict Monotonicity of Isosceles Family along deformation]
    For $0<\beta \leq \pi / 2$, let $T_\beta=\triangle O A B$ be the isosceles spherical triangle satisfying

$$|O A|=\frac{\pi}{2}, \quad |O B|=|A B|, \quad \angle A O B=\beta .
$$

Then the function

$$
\beta \longmapsto \mu_1^N\left(T_\beta\right)
$$

is strictly decreasing on $(0, \pi / 2)$. Consequently,

$$
\mu_1^N\left(T_\beta\right)>6, \quad 0<\beta<\frac{\pi}{2},
$$

and
$
\lim _{\beta \uparrow \pi / 2} \mu_1^N\left(T_\beta\right)=6 .
$
\end{theorem}

\begin{proof}
  Choose geodesic polar coordinates $(r, \theta)$ centered at $O$, with the side $O A$ given by $\theta=0$:

\begin{equation}\label{geodesicpolarcoordinates}
X(r, \theta)=(\sin r \cos \theta, \sin r \sin \theta, \cos r) .
\end{equation}

The midpoint of $O A$ is $M=\frac{1}{\sqrt{2}}(1,0,1).$ The great circle containing the symmetry axis $B M$ lies in the plane $\{X_1=X_3\}.$ In polar coordinates, its equation is therefore

$$
\sin r \cos \theta=\cos r
$$

or equivalently

$$
\cot r=\cos \theta .
$$

Define $\rho(\theta):= \cot^{-1}(\cos \theta) .$ Then

$$
B=X\left(\ell_\beta, \beta\right), \quad \ell_\beta:=O B=\rho(\beta),
$$

and

$$
\cot \ell_\beta=\cos \beta, \quad \ell_\beta=\arctan (\sec \beta) .
$$  
Let $U_\beta:=\triangle O M B$ be one half of $T_\beta$. In these coordinates,

$$
U_\beta=\{(r, \theta): 0<\theta<\beta, 0<r<\rho(\theta)\} .
$$

Its boundary consists of

$$
O M=\{\theta=0\}, \quad O B=\{\theta=\beta\}, \quad B M=\{r=\rho(\theta)\} .
$$

Notice that the geodesics containing $O M$ and $B M$ are independent of $\beta$. As $\beta$ increases, only the Neumann side $O B$ moves.
By \cref{thm: antisymmetry of isosceles eigenfunction}, the anti-symmetry of a first Neumann eigenfunction across $B M$, $\mu(\beta):=\mu_1^N\left(T_\beta\right)$ is the first eigenvalue of the mixed problem on $U_\beta$:

$$
\begin{cases}-\Delta v_\beta=\mu(\beta) v_\beta & \text { in } U_\beta, \\ v_\beta=0 & \text { on } B M, \\ \partial_\nu v_\beta=0 & \text { on } O M \cup O B .\end{cases}
$$

The maximum principle implies that $v_\beta$ does not change its sign. Choose the normalization

$$
\int_{U_\beta} v_\beta^2 d A=1
$$
and the sign so that $v_\beta>0$ in $U_\beta .$ The first mixed eigenvalue is simple, so $\mu(\beta)$ and $v_\beta$ depend differentiably on $\beta$. Parametrize $O B$ by arclength from $O$:

$$
\gamma_\beta(s)=X(s, \beta), \quad 0<s<\ell_\beta .
$$

Differentiating with respect to $\beta$,

$$
\partial_\beta \gamma_\beta(s)=\partial_\theta X(s, \beta).
$$

Along the side $\theta=\beta$, the outward unit normal vector to $\partial U_\beta$ is $\nu=\frac{1}{\sin s} \partial_\theta,$ hence the outward normal velocity is

$$
V_\nu(s)=\left\langle\partial_\beta \gamma_\beta(s), \nu\right\rangle=\left\langle\partial_\theta, \frac{1}{\sin s} \partial_\theta\right\rangle=\sin s .
$$

The sides $O M$ and $B M$ have zero normal velocity.
The mixed Hadamard formula (see \cite[Theorem~A.1]{AnoopAshokKesavan2021} and \cite[Section~3]{Berge2024}) therefore gives

$$
\mu^{\prime}(\beta)=\int_{O B}\left(\left|\nabla_\tau v_\beta\right|^2-\mu(\beta) v_\beta^2\right) V_\nu d s.
$$
Writing $y(s):=v_\beta(s, \beta),$ we obtain

$$
\mu^{\prime}(\beta)=\int_0^{\ell_\beta}\left(y^{\prime}(s)^2-\mu(\beta) y(s)^2\right) \sin s d s.
$$

Set

$$
I_\beta:=\int_0^{\ell_\beta}\left(y^{\prime}(s)^2-\mu(\beta) y(s)^2\right) \sin s d s.
$$

It remains to show $I_\beta<0 .$ 

We first show 
\begin{lemma}[Angular Derivative]\label{positivity of the angular derivative q}
    $q:=\partial_\theta v_\beta >0 .$
\end{lemma}
\begin{proof}[Proof of \cref{positivity of the angular derivative q}]
    The vector field $\partial_\theta$ is generated by rotations about $O$, hence it is a Killing field. Therefore it commutes with the spherical Laplacian:

$$
\left[\Delta, \partial_\theta\right]=0
$$

and consequently

$$
-\Delta q=\mu(\beta) q \quad \text { in } U_\beta .
$$

On the radial sides $O M$ and $O B$, the Neumann condition for $v_\beta$ gives $q=0 .$ We next determine the sign of $q$ on $B M$. Since

$$
B M=\{r=\rho(\theta)\}, \quad \cot \rho(\theta)=\cos \theta,
$$

differentiation gives

$$
-\csc ^2 \rho(\theta) \rho^{\prime}(\theta)=-\sin \theta,
$$

and hence

$$
\rho^{\prime}(\theta)=\sin ^2 \rho(\theta) \sin \theta>0 .
$$

The Dirichlet condition on $B M$ is

$$
v_\beta(\rho(\theta), \theta)=0 .
$$

Differentiating tangentially,

$$
\left(v_\beta\right)_r \rho^{\prime}(\theta)+\left(v_\beta\right)_\theta=0 .
$$

Thus

$$
q=-\rho^{\prime}(\theta)\left(v_\beta\right)_r \quad \text { on } B M .
$$

Because $v_\beta>0$ in $U_\beta$ and vanishes on $B M$, the Hopf boundary lemma gives

$$
\partial_\nu v_\beta<0 \quad \text { on the open side } B M .
$$

The domain lies on the side $r<\rho(\theta)$, so the outward normal vector has positive $r$-component. It follows that

$$
\left(v_\beta\right)_r<0 \quad \text { on } B M .
$$

Since $\rho^{\prime}>0$,

$$
q>0 \quad \text { on the open side } B M .
$$

We claim that $q \geq 0$ in $U_\beta.$ Suppose otherwise and let $q^{-}:=\max \{-q, 0\} .$ Since $q=0$ on $O M \cup O B$ and $q>0$ on $B M$, one has $q^{-} \in H_0^1\left(U_\beta\right) .$ Testing
$$
-\Delta q=\mu(\beta) q
$$
against $q^{-}$gives
$$
\int_{U_\beta}\left|\nabla q^{-}\right|^2 d A=\mu(\beta) \int_{U_\beta}\left(q^{-}\right)^2 d A.
$$

On the other hand,
$$
\mu(\beta)<\lambda_1^D\left(U_\beta\right),
$$
where $\lambda_1^D\left(U_\beta\right)$ is the first pure Dirichlet eigenvalue. Indeed, the mixed admissible space strictly contains $H_0^1\left(U_\beta\right)$; equality would force a first Dirichlet eigenfunction also to satisfy the Neumann condition on the open sides $O M$ and $O B$, contradicting the Hopf lemma.
Thus $q^{-} \not \equiv 0$ would have Dirichlet Rayleigh quotient
$$
\frac{\int\left|\nabla q^{-}\right|^2}{\int\left(q^{-}\right)^2}=\mu(\beta)<\lambda_1^D\left(U_\beta\right)
$$
which is impossible. Hence $q \geq 0 .$ Since $q \not \equiv 0$, the strong maximum principle yields
$$
q>0 \quad \text { in } U_\beta .
$$
The lemma follows.
\end{proof}
Along the side $O B=\{\theta=\beta\}$, one has $q=0$, while $q>0$ inside. Hopf's lemma therefore gives
$$
\partial_\nu q<0 \quad \text { on the open side } O B .
$$

Because $\partial_\nu=\frac{1}{\sin s} \partial_\theta$ there,
$$
q_\theta(s, \beta)<0 .
$$

Equivalently,
$$
v_{\beta, \theta \theta}(s, \beta)<0, \quad 0<s<\ell_\beta.
$$
We now determine the sign of the integrand in the Hadamard integral $I_{\beta}$. Along $O B$, the eigenvalue equation becomes
$$
y^{\prime \prime}+\cot s y^{\prime}+\frac{1}{\sin ^2 s} v_{\beta, \theta \theta}(s, \beta)+\mu(\beta) y=0
$$
Hence
$$
y^{\prime \prime}+\cot s y^{\prime}+\mu(\beta) y=-\frac{v_{\beta, \theta \theta}(s, \beta)}{\sin ^2 s} .
$$

Integrating by parts,
$$
\begin{aligned}
I_\beta & =\int_0^{\ell_\beta}\left(y^{\prime 2}-\mu(\beta) y^2\right) \sin s d s \\
& =-\int_0^{\ell_\beta} y\left(y^{\prime \prime}+\cot s y^{\prime}+\mu(\beta) y\right) \sin s d s
\end{aligned}
$$

The boundary term vanishes: at $s=0$ because $\sin s=0$, and at $s=\ell_\beta$ because $B \in B M$ and therefore
$$
y\left(\ell_\beta\right)=v_\beta(B)=0 .
$$

Substituting the equation for $y$,
$$
I_\beta=\int_0^{\ell_\beta} \frac{y(s) v_{\beta, \theta \theta}(s, \beta)}{\sin s} d s
$$
Now
$$
y(s)>0 \quad \text { for } 0<s<\ell_\beta,
$$
because $v_\beta$ is the positive first mixed eigenfunction, while
$$
v_{\beta, \theta \theta}(s, \beta)<0 .
$$

Therefore
$$
I_\beta<0 .
$$

Since
$
\mu^{\prime}(\beta)=I_\beta,
$
we conclude that
$$
\mu^{\prime}(\beta)<0 \quad \text { for every } 0<\beta<\frac{\pi}{2}.
$$

Finally, as $\beta \uparrow \pi / 2$, the triangle $T_\beta$ converges to the equilateral spherical triangle with side length $\pi / 2$, whose first nonzero Neumann eigenvalue is 6. Hence
$$
\mu_1^N\left(T_\beta\right)=6-\int_\beta^{\pi / 2} \mu^{\prime}(\sigma) d \sigma>6.
$$
This proves both strict monotonicity and the desired lower bound along the isosceles family.
\end{proof}

\printbibliography

@article{Kroger1992SpectralGap,
  author    = {Kröger, Pawel},
  title     = {On the spectral gap for compact manifolds},
  journal   = {J. Differential Geom.},
  volume    = {36},
  number    = {2},
  pages     = {315--330},
  year      = {1992},
  publisher = {Lehigh University},
  url       = {https://projecteuclid.org/journals/journal-of-differential-geometry/volume-36/issue-2/On-the-spectral-gap-for-compact-manifolds/10.4310/jdg/1214448744.full}
}

@article{LaugesenSiudeja2010,
  author  = {Laugesen, Richard Snyder and Siudeja, Bart{\l}omiej Andrzej},
  title   = {Minimizing {Neumann} Fundamental Tones of Triangles:
             An Optimal {Poincar\'e} Inequality},
  journal = {Journal of Differential Equations},
  volume  = {249},
  number  = {1},
  pages   = {118--135},
  year    = {2010},
  doi     = {10.1016/j.jde.2010.03.020}
}

@misc{Hatcher2025,
  author        = {Hatcher, Lawford},
  title         = {Hot Spots in Domains of Constant Curvature},
  year          = {2025},
  eprint        = {2508.13353},
  archivePrefix = {arXiv},
  primaryClass  = {math.AP}
}

@article {BenAndrews2013,
    AUTHOR = {Andrews, Ben and Clutterbuck, Julie},
     TITLE = {Sharp modulus of continuity for parabolic equations on
              manifolds and lower bounds for the first eigenvalue},
   JOURNAL = {Anal. PDE},
  FJOURNAL = {Analysis \& PDE},
    VOLUME = {6},
      YEAR = {2013},
    NUMBER = {5},
     PAGES = {1013--1024},
      ISSN = {2157-5045,1948-206X},
   MRCLASS = {35R01 (35B65 35K59 35P15)},
  MRNUMBER = {3125548},
MRREVIEWER = {Rodica\ Luca},
       DOI = {10.2140/apde.2013.6.1013},
       URL = {https://doi.org/10.2140/apde.2013.6.1013},
}

@article { Rohleder2017,
    AUTHOR = {Lotoreichik, Vladimir and Rohleder, Jonathan},
     TITLE = {Eigenvalue inequalities for the {L}aplacian with mixed
              boundary conditions},
   JOURNAL = {J. Differential Equations},
  FJOURNAL = {Journal of Differential Equations},
    VOLUME = {263},
      YEAR = {2017},
    NUMBER = {1},
     PAGES = {491--508},
      ISSN = {0022-0396,1090-2732},
   MRCLASS = {35P15 (35J05)},
  MRNUMBER = {3631314},
       DOI = {10.1016/j.jde.2017.02.043},
       URL = {https://doi.org/10.1016/j.jde.2017.02.043},
}

@article{AnoopAshokKesavan2021,
  author  = {Anoop, T. V. and Ashok Kumar, K. and Kesavan, S.},
  title   = {A shape variation result via the geometry of eigenfunctions},
  journal = {Journal of Differential Equations},
  volume  = {298},
  pages   = {430--462},
  year    = {2021},
  doi     = {10.1016/j.jde.2021.07.001}
}

@article{Berge2024,
  author  = {Berge, Stine Marie},
  title   = {Kuttler--Sigillito inequalities and
             Rellich--Christianson identity},
  journal = {Studia Mathematica},
  volume  = {277},
  number  = {1},
  pages   = {45--63},
  year    = {2024},
  doi     = {10.4064/sm230716-9-6}
}

@article{ChenGuiYao2026,
  author  = {Chen, Hongbin and Gui, Changfeng and Yao, Ruofei},
  title   = {Uniqueness of critical points of the second {Neumann} eigenfunctions on triangles},
  journal = {Inventiones Mathematicae},
  year    = {2026},
  volume  = {244},
  pages   = {299--353},
  doi     = {10.1007/s00222-025-01398-x},
  url     = {https://doi.org/10.1007/s00222-025-01398-x}
}

@article {PayneWeinberger1960,
    AUTHOR = {Payne, L. E. and Weinberger, H. F.},
     TITLE = {An optimal {P}oincar\'e{} inequality for convex domains},
   JOURNAL = {Arch. Rational Mech. Anal.},
  FJOURNAL = {Archive for Rational Mechanics and Analysis},
    VOLUME = {5},
      YEAR = {1960},
     PAGES = {286--292},
      ISSN = {0003-9527},
   MRCLASS = {35.00},
  MRNUMBER = {117419},
MRREVIEWER = {I.\ Stakgold},
       DOI = {10.1007/BF00252910},
       URL = {https://doi-org.lib-proxy.fullerton.edu/10.1007/BF00252910},
}

@article {BakryQian2000,
    AUTHOR = {Bakry, Dominique and Qian, Zhongmin},
     TITLE = {Some new results on eigenvectors via dimension, diameter, and
              {R}icci curvature},
   JOURNAL = {Adv. Math.},
  FJOURNAL = {Advances in Mathematics},
    VOLUME = {155},
      YEAR = {2000},
    NUMBER = {1},
     PAGES = {98--153},
      ISSN = {0001-8708,1090-2082},
   MRCLASS = {58J50 (53C21)},
  MRNUMBER = {1789850},
MRREVIEWER = {Olivier\ Druet},
       DOI = {10.1006/aima.2000.1932},
       URL = {https://doi-org.lib-proxy.fullerton.edu/10.1006/aima.2000.1932},
}

@article {JudgeMondal2020,
    AUTHOR = {Judge, Chris and Mondal, Sugata},
     TITLE = {Euclidean triangles have no hot spots},
   JOURNAL = {Ann. of Math. (2)},
  FJOURNAL = {Annals of Mathematics. Second Series},
    VOLUME = {191},
      YEAR = {2020},
    NUMBER = {1},
     PAGES = {167--211},
      ISSN = {0003-486X,1939-8980},
   MRCLASS = {35P05 (35B38 35J05 35J25 58J50)},
  MRNUMBER = {4045963},
MRREVIEWER = {Luigi\ Provenzano},
       DOI = {10.4007/annals.2020.191.1.3},
       URL = {https://doi-org.lib-proxy.fullerton.edu/10.4007/annals.2020.191.1.3},
}

@article{JudgeMondal2022Erratum,
  author    = {Judge, Chris and Mondal, Sugata},
  title     = {Erratum: Euclidean triangles have no hot spots},
  journal   = {Annals of Mathematics},
  volume    = {195},
  number    = {1},
  pages     = {337--362},
  year      = {2022},
  publisher = {Department of Mathematics of Princeton University},
  doi       = {10.4007/annals.2022.195.1.5},
  url       = {https://annals.math.princeton.edu/2022/195-1/p05}
}

@article{Cheng1976,
author = {Cheng, Shiu-Yuen},
journal = {Commentarii mathematici Helvetici},
pages = {43-56},
title = {Eigenfunctions and Nodal Sets},
url = {http://eudml.org/doc/139642},
volume = {51},
year = {1976},
}

@article{FreitasKennedy2025,
  author  = {Pedro Freitas and James B. Kennedy},
  title   = {On domain monotonicity of Neumann eigenvalues of convex domains},
  journal = {Proceedings of the American Mathematical Society},
  volume  = {153},
  number  = {12},
  pages   = {5315--5328},
  year    = {2025},
  doi     = {10.1090/proc/17394}
}

@article{BW99,
  author  = {Krzysztof Burdzy and Wendelin Werner},
  title   = {A counterexample to the ``hot spots'' conjecture},
  journal = {Annals of Mathematics},
  volume  = {149},
  number  = {1},
  pages   = {309--317},
  year    = {1999},
  doi     = {10.2307/121027}
}

@article{BB99,
  author  = {Rodrigo Ba{\~n}uelos and Krzysztof Burdzy},
  title   = {On the ``hot spots'' conjecture of {J}. {R}auch},
  journal = {Journal of Functional Analysis},
  volume  = {164},
  number  = {1},
  pages   = {1--33},
  year    = {1999},
  doi     = {10.1006/jfan.1999.3397}
}

@article{AB04,
  author  = {Rami Atar and Krzysztof Burdzy},
  title   = {On Neumann eigenfunctions in lip domains},
  journal = {Journal of the American Mathematical Society},
  volume  = {17},
  number  = {2},
  pages   = {243--265},
  year    = {2004},
  doi     = {10.1090/S0894-0347-04-00453-9}
}

@article{LiEdelen2022,
author = {Edelen, Nick and Li, Chao},
title = {Regularity of Free Boundary Minimal Surfaces in Locally Polyhedral Domains},
journal = {Communications on Pure and Applied Mathematics},
volume = {75},
number = {5},
pages = {970-1031},
doi = {https://doi.org/10.1002/cpa.22039},
url = {https://onlinelibrary.wiley.com/doi/abs/10.1002/cpa.22039},
eprint = {https://onlinelibrary.wiley.com/doi/pdf/10.1002/cpa.22039},
year = {2022}
}
\end{document}